\documentclass[a4paper,12pt]{article}
\usepackage{amsmath, amsthm, amstext}
\usepackage{amssymb, array, amsfonts}
\usepackage[english]{babel}
\usepackage{bbding}

\numberwithin{equation}{section}

\def \de{\delta}
\def \er{\varepsilon}
\def \eps{\varepsilon}
\def \la{\lambda}
\def \si{\sigma}

\def \sk{\mbox{ }}
\renewcommand{\l}{\left}
\renewcommand{\r}{\right}

\def \G{\Gamma}

\def \O{\Omega}

\def \C{\mathbb{C}}
\def \N{\mathbb{N}}
\def \R{\mathbb{R}}
\def \Z{\mathbb{Z}}
\def \T{\mathbb{T}}

\def\dd{\partial}

\newcommand{\beq}{\begin{equation}}
\newcommand{\eeq}{\end{equation}}

\def\dom{\operatorname{Dom}}

\def\im{\operatorname{Im}}

\newcommand{\eqdef}{\stackrel{\rm def}{=\kern-3.6pt=}}
\newcommand{\<}{\langle}
\renewcommand{\>}{\rangle}

\theoremstyle{plain}
\newtheorem{theorem}{\bf Theorem}[section]
\newtheorem{lemma}[theorem]{\bf Lemma}

\newtheorem{cor}[theorem]{\bf Corollary}

\theoremstyle{definition}

\theoremstyle{remark}
\newtheorem{rem}[theorem]{\bf Remark}

\renewcommand{\le}{\leqslant}
\renewcommand{\ge}{\geqslant}
\newcommand{\loc}{\mathop{\mathrm{loc}}\nolimits}

\renewcommand{\qed}{\vrule height7pt width5pt depth0pt}

\title{Absolute continuity of the spectrum of the periodic Schr\"odinger operator in a layer and in a smooth cylinder}
\date{}
\author{N.~Filonov \and I.~Kachkovskiy\thanks{The first author was supported by RFBR grant 08-01-00209.}}
\begin{document}
\maketitle

\begin{abstract}
We consider the Schr\"odinger operator
$H = -\Delta + V$ in a layer or in a $d$-dimensional cylinder.
The potential $V$ is assumed to be periodic with respect to some lattice.
We establish the absolute continuity of $H$, assuming $V \in L_{p, \loc}$, where $p$ is a real number greater than $d/2$ in the case of a layer, and $p > \max (d/2, d-2)$ 
for the cylinder.
\end{abstract}

\section{Introduction}

Let $M$ be a smooth $k$-dimensional compact Riemannian manifold, let also
$$
\Xi = M \times \R^m,\sk\quad d := \dim \Xi = k+m.
$$
We are interested in the type of the spectrum of the Schr\"odinger operator
$H = - \Delta + V$ in a cylinder $\Xi$.
The function $V$ is supposed to be periodic.
If $M$ is a manifold with boundary, we study the operator $H$
with various boundary conditions at
$\partial \Xi = \partial M \times \R^m$.
We are going to prove that, under some assumptions on $V$,
the spectrum of $H$ is absolutely continuous
(see Theorems \ref{t1} and \ref{t2} below).

The points of $\Xi$ are denoted by
 $(x, y)$, $x \in M$, $y \in \R^m$.
Let $\Gamma$ be a lattice in $\R^m$,
\beq
\label{lattice}
\G = \left\{ l = \sum_{j=1}^m l_j b_j, \quad l_j \in \Z \right\} ,
\eeq
where $\{b_j\}_{j=1}^m$ is a basis of $\R^m$. 
Assume that $V$ 
is periodic over the "longitudinal" variables:
\beq
\label{v}
V(x,y+l)=V(x,y),\quad x\in M,\sk y\in \R^m,\sk l\in\Gamma.
\eeq
Thanks to $V$ being periodic, it is enough to know $V$ on $M\times\Omega$, where
\beq
\label{yach}
 \O = \left\{ y = \sum_{j=1}^m y_j b_j, \quad y_j \in [0, 1) \right\}
\eeq
is an elementary cell of $\G$.

Let us introduce the reader to the main results regarding absolute continuity of $H$.
Usually, in the sufficient conditions it is assumed that the potential $V$ belongs to
$L_p (M \times \O)$ or to a Lorentz space
$L_{p, \infty}^0 (M \times \O)$. 
We recall that if $N$ is a set of finite measure, then $L_p (N) \subset L_{p, \infty}^0 (N) \subset L_{p - \er} (N)$
for all $\er > 0$.

The two-dimensional case, $d=2$
($\Xi$ is a whole plane or a strip), has been studied in much detail.
In \cite{BSu1, SuSh2, SoSh}, the absolute continuity of $H$
is proved for $V\in L_p$, $p>1$.
From now on, we consider only $d\geqslant 3$.

The case of $k=0$, corresponding to the operator in the whole space, is also well studied.
In \cite{Shen}, the absolute continuity is established in the "critical"  
case $V \in L_{d/2,\infty}^0 (\O)$ for all $d\geqslant 3$
(see also \cite{D}).
In \cite{Su}, the case $k = 1$ ($M$ is a line segment, $\Xi$ is a plane-parallel layer) 
is studied, and for
$V \in L_{p, \infty}^0 (M \times \O)$, where
$p = \max (d/2, d-2)$, the absolute continuity of $H$ is obtained. The author also considers the third type boundary condition.
Finally, the case $k\geqslant 2$ is studied in \cite{KF}, and
it is established that $H$ is absolutely continuous if $V \in L_{d-1} (M \times \O)$.

In the present paper, we prove (see Theorem \ref{t1} below) 
the absolute continuity of $H$ with $V \in L_p (M \times \O)$
for all $p>d/2$ in the following cases:
1) $\dd M = \emptyset$;
2) $M$ is a line segment, $k=1$;
3) $d=3$ or 4. 
If $M$ is a manifold with boundary, $k>1$, and $d>4$,
we obtain only $V \in L_{d-2} (M \times \O)$ as a sufficient condition.
In the case $k=1$ we also consider the third type boundary condition (see Theorem \ref{t2}).

All mentioned results are obtained using the Thomas scheme  \cite{Th}, its key point is to study the operator family
$$
H(\xi)=-\Delta_x+(-i\nabla_y+\xi)^{*}(-i\nabla_y+\xi)+V(x,y),
$$ 
where $\xi$ is called quasimomentum.
To obtain the resolvent estimates for the free operator
$H_0 (\xi)$, corresponding to $V=0$, we use the spectral cluster estimates from \cite{So} (the idea of using these estimates arose in \cite{Shen}).

\section{Formulation of the result}
Let $M$ be a compact smooth Riemannian manifold with or without boundary, $\dim M = k$. 
Consider a $d$-dimensional cylinder
$$
\Xi = M \times \R^m,\sk\quad d=k+m\geqslant 3.
$$
Let $\Gamma$ be a lattice \eqref{lattice}, 
let $\O$ be a cell \eqref{yach}, and let
$V(x,y)$ be a real-valued function, 
satisfying \eqref{v}.
Assume that
\begin{equation}
\label{245}
V \in L_{d/2} (M \times \O) .
\end{equation}
Consider the following quadratic form in $L_2(\Xi)$:
\beq
\label{hn}
h [u,u]=\int \limits_{\Xi}\l(|\nabla u(x,y)|^2 + V(x,y)|u(x,y)|^2\r)\,dx\,dy,
\quad \dom h=H^1(\Xi).
\eeq
If $M$ has a boundary, $\dd M \neq \emptyset$,
then we denote \eqref{hn} by $h_N$.
In this case we are also going to study a form
$h_D = h_N \mid_{H^1_0 (\Xi)}$.

It is well known that, assuming \eqref{245}, the form $h$ 
(resp. $h_N$, $h_D$) is closed and semi-bounded from below.
In $L_2(\Xi)$, it corresponds to a semi-bounded operator
$H$ (resp. $H_N$, $H_D$), which is called the Schr\"odinger 
operator in $\Xi$ (resp. the Schr\"odinger operator with Dirichlet or Neumann boundary conditions).

\begin{theorem}
\label{t1}
Let $M$ be a compact smooth Riemannian manifold with or without boundary,
$\dim M = k$, $\Xi = M \times \R^m$, $d=k+m\geqslant 3$.
Let $\Gamma$ be a lattice \eqref{lattice},
let $V$ be a real-valued $\G$-periodic function in $\Xi$.
Assume that
$V \in L_p (M \times \O)$, where
\begin{itemize}
\item $p > d/2$, if $\dd M = \emptyset$;
\item $p > d/2$, if $\dd M \neq \emptyset$ and $k=1$ {\rm(}$M$ is a line segment{\rm)};
\item $p > d/2$, if $\dd M \neq \emptyset$ and $d=3$ or $d=4$;
\item $p > d-2$, if $\dd M \neq \emptyset$ and $d \ge 5$.
\end{itemize}
Then the spectra of $H$ {\rm(}$\dd M = \emptyset${\rm)}, 
$H_N$ and $H_D$ {\rm(}$\dd M \neq \emptyset${\rm)} are absolutely continuous.
\end{theorem}

In the case of a layer ($M$ is a line segment), Suslina's result \cite{Su} 
(see Theorem \ref{su} below)
allows us to consider the case of the third type boundary condition.
Let $k=1$, $\Xi = [0, a] \times \R^m$, 
let also $\si$ be a real $\G$-periodic function on 
$\dd\Xi = \{0; a\} \times \R^m$.
Consider a quadratic form
\begin{eqnarray}
\nonumber
h_{\sigma}[u,u]=
\int \limits_{\Xi} \l( |\nabla u(x,y)|^2 + V(x,y) |u(x,y)|^2 \r)\,dx\,dy \\
\label{h3}
+ \int\limits_{\R^m} \l( \sigma(a,y)|u(a,y)|^2 - \sigma(0,y)|u(0,y)|^2 \r)\,dy ,
\qquad \dom h_\si = H^1(\Xi).
\end{eqnarray}
If $\sigma\in L_m(\{0,a\}\times\O)$, then the form \eqref{h3} 
is closed and semi-bounded from below (see \cite{SuSh}). 
In the case $\si = 0$ the form $h_\si$ coincides with $h_N$.

\begin{theorem}
\label{t2}
Let $\Xi = [0, a] \times \R^m$, $d = m+1 \ge 3$,
let $\G$ be a lattice \eqref{lattice}.
Let $V$ be a $\G$-periodic function on $\Xi$, 
$V \in L_p ([0, a] \times \O)$ with $p > d/2$.
Let $\si$ be a  $\G$-periodic function on $\dd\Xi$, satisfying
\begin{equation}
\label{sigma}
\si \in L_q (\{0,a\}\times\O), \quad \text{where} \quad
q=2 \text{ for } d=3, \quad q=2d-2 \text{ for } d\ge 4.
\end{equation}
Then the spectrum of the Schr\"odinger operator $H_\si$, 
corresponding to the form \eqref{h3}, is absolutely continuous.
\end{theorem}

\begin{rem}
Theorem \ref{t1} can be reformulated in the matrix case.
Let $V$ be an $(n\times n)$-matrix-valued function on $\Xi$
such that $V(x, y)^* = V(x, y)$, \eqref{v} holds, and
$V \in L_p (M\times \O)$, $p > d/2$.
The quadratic form
$$
h [u,u] = \int \limits_{\Xi}\l(|\nabla u(x,y)|^2 
+ \< V(x,y) u(x,y)\,,\,u(x,y)\>\r)\,dx\,dy
$$
is closed and semi-bounded on the domains
$H^1(\Xi, \C^n)$ and $H^1_0(\Xi, \C^n)$.
These forms correspond to the self-adjoint operators
$H$, $H_N$, $H_D$ in $L_2 (\Xi, \C^n)$.
In the cases of a manifold without boundary, a layer, and 3- and 4-dimensional cylinders, the spectra of such operators are absolutely continuous. In the case of a $d$-dimensional cylinder, $d > 4$, 
the spectra of $H_N$ and $H_D$ are absolutely continuous whenever
$V \in L_p (M\times \O)$, $p > d-2$.
The proof of Theorem \ref{t1} is valid for the matrix case without changes. A matrix analog of Theorem \ref{t2} can also be obtained.
\end{rem}

It is convenient for us to interpret $\Omega$ as an 
$m$-dimensional torus $\T =\R^m / \Gamma$. 
Let us introduce an additional parameter $\xi \in \C^m$, and consider the following quadratic forms.
In the case of a manifold without boundary let
\begin{eqnarray}
\label{hnk}
h(\xi) [v,v] = \int\limits_{M\times \Omega} ( |\nabla_x v|^2
+ \langle (\nabla_y+i\xi)v,(\nabla_y+i\overline{\xi})v \rangle 
+ V |v|^2 )\, dx\,dy, \\ 
\nonumber
\dom h (\xi) = H^1 (M \times \T) . 
\end{eqnarray}
If $\dd M \neq \emptyset$, then the form \eqref{hnk} will be denoted by $h_N (\xi)$, and let also
$$
h_D (\xi) = h_N (\xi) \mid_{H^1_0 (M \times \T)}.
$$
In the case of a layer, $\Xi = [0, a] \times \R^m$, consider also a form
\begin{eqnarray*}
h_\si (\xi) [v,v] = h_N (\xi) [v,v] +
\int\limits_{\Omega} \l( \sigma(a,y)|v(a,y)|^2 - \sigma(0,y)|v(0,y)|^2 \r) dy,\\ 
\dom h_{\sigma} (\xi) = H^1([0, a] \times \T).
\end{eqnarray*}
These forms are sectorial (the definition and main properties of sectoriality can be found in \cite[Ch. VI, VII]{K}),
and they correspond to analytic operator families
$H(\xi)$, $H_N(\xi)$, $H_D(\xi)$, and $H_\si(\xi)$ respectively.
For real $\xi$, these operators are self-adjoint.

Let $b_1$ be the first vector in the basis of $\G$.
The conditions on the potential are dilatation-invariant, so we can assume $|b_1|=1$. 

\begin{theorem}
\label{t3}
Suppose the conditions of Theorem {\rm\ref{t1}} or Theorem {\rm\ref{t2}} are satisfied.
Then, for every $\lambda\in\C$ and $\xi \in \R^m$, $\xi \perp b_1$, 
there exists $\tau_0$ such that for $|\tau| > \tau_0$
the operator $\left( H ((\pi + i \tau) b_1 + \xi) - \la I \right)$
is invertible and
\begin{equation}
\label{18}
\left\| \left( H ((\pi + i \tau) b_1 + \xi) - \la I \right)^{-1}
\right\| \le C |\tau|^{-1}.
\end{equation}
\end{theorem}

We prove this Theorem in \S 4.
In a standard way (see, for example, \cite{BSu} or \cite{RS})
Theorem \ref{t3} implies Theorems \ref{t1} and \ref{t2}.

\section{Spectral cluster estimates}

For a self-adjoint operator $P$, we denote by 
$E_k (P) = E_P ([(k-1)^2;k^2))$ its spectral projector onto a subspace, 
corresponding to an interval $[(k-1)^2;k^2)$. 
The following Theorem is proved in \cite{So}.

\begin{theorem}
\label{mainit}
Let $N$ be a compact $C^{\infty}$-smooth d-dimensional Riemannian 
manifold without boundary, let $P$ be an elliptic second-order 
differential operator on $N$ with positive-definite symbol. Then 
$$
\| E_k (P) f \|_{L_2(N)} \leqslant 
C k^{d(1/p - 1/2) - 1/2} \| f \|_{L_p(N)},
\qquad f\in L_p(N), \quad 1 \leqslant p \leqslant \frac{2(d+1)}{d+3} .
$$
\end{theorem}

By duality, this yields

\begin{cor}
\label{sl22}
Under the assumptions of Theorem {\rm\ref{mainit}}, the following inequality holds:
\beq
\label{maini}
\|E_k (P) f\|_{L_q(N)} \leqslant C k^{d(1/2-1/q)-1/2} \|f\|_{L_2(N)},
\qquad f\in L_2(N), \quad \frac{2(d+1)}{d-1}\leqslant q\leqslant +\infty.
\eeq
\end{cor}

\begin{theorem}
\label{prod_t}
Let $N_0$ be a compact smooth Riemannian manifold without boundary,
$\dim N_0 = d-1$. Let $P_0$ be a second-order elliptic differential operator on 
$N_0$ with positive-definite symbol.
Consider an elliptic operator
$P=1\otimes P_0-\frac{d^2}{dx^2}\otimes 1$ 
on a manifold $N=[0,a]\times N_0$
{\rm(}$x$ denotes a local coordinate on $[0,a]${\rm)}.
Then, for $P$ on $N$ with either Dirichlet or Neumann boundary conditions, the estimate \eqref{maini} holds.
\end{theorem}

\begin{proof}
We shal give proof for the Dirichlet problem, the Neumann case is analogous. The 
statement of Theorem is invariant with respect to dilatations over $x$, so we 
can assume $a=\pi$. In this case, the spectral projector $E_k$ of $P$ is an integral operator
with kernel
\beq
\label{kker}
K(x,x',y,y')=\sum\limits_{j^2+\lambda_n\in [(k-1)^2;k^2)}\frac{2}{\pi}\sin(jx)\sin(jx')\varphi_n(y)\overline{\varphi}_n(y'),
\eeq
where $\{\lambda_n\}$, $\{\varphi_n\}$ are eigenvalues and eigenfunctions
of $P_0$. 
We introduce three operators: an operator $\widetilde{E}_k$, 
acting on functions from $L_2([0,2\pi]\times N_0)$ 
as an integral operator with the same kernel \eqref{kker},
an operator of zero extension
$T\colon L_2(N)\to L_2([0,2\pi]\times N_0)$,
and a restriction operator $S\colon L_{q}([0,2\pi]\times N_0)\to L_{q}(N)$.
Obviously, $E_k=S\widetilde{E}_k T$. 
Furthermore, 
$\widetilde{E}_k=\frac{1}{2\pi}(\widetilde{E}_k^{(1)}-\widetilde{E}_k^{(2)})$, 
where $\widetilde{E}_k^{(1)}$ and $\widetilde{E}_k^{(2)}$ 
are integral operators with kernels
\begin{eqnarray*}
K^{(1)}(x,x',y,y')=\sum\limits_{j^2+\lambda_n\in [(k-1)^2;k^2)}(e^{ij(x-x')}+e^{-ij(x-x')})\varphi_n(y)\overline{\varphi}_n(y'), \\
K^{(2)}(x,x',y,y')=\sum\limits_{j^2+\lambda_n\in [(k-1)^2;k^2)}(e^{ij(x+x')}+e^{-ij(x+x')})\varphi_n(y)\overline{\varphi}_n(y').
\end{eqnarray*}
The operator $\widetilde{E}_k^{(1)}$ is a spectral projector of $-\frac{d^2}{dx^2}\otimes 1+1\otimes P_0$ 
on $[0,2\pi]\times N_0$ with periodic boundary conditions over $x$. 
The last operator is an elliptic operator on a manifold $S^1\times N_0$ without boundary, 
and it satisfies \eqref{maini}. 
Similarly, \eqref{maini} holds for $\widetilde{E}_k^{(2)}$, 
and so for $\widetilde{E}_k$ and $E_k$.

The proof for the Neumann case can be obtained by replacing $\sin(jx)$ with $\cos(jx)$, in this case $\widetilde{E}_k=\frac{1}{2\pi}(\widetilde{E}_k^{(1)}+\widetilde{E}_k^{(2)})$.
\end{proof}
In \cite{SmSo}, the following result is proved.

\begin{theorem}
\label{mainit2}
Let $N$ be a compact smooth Riemannian manifold with boundary, $\dim N=d\geqslant 3$. 
Let $P$ be an elliptic second-order differential operator on $N$ with positive-definite symbol and with
Dirichlet or Neumann boundary conditions. Then, for
\beq
\label{star2}
5\leqslant q\leqslant\infty\text{, if } d=3 ; \quad
4\leqslant q\leqslant \infty\text{, if } d\geqslant 4 ,
\eeq
the estimate \eqref{maini} holds. For
$$
2\leqslant q\leqslant 4, \qquad d \geqslant 4,
$$
the estimate is replaced with
\beq
\label{q24}
\|E_k f\|_{L_{q}(N)}\leqslant 
Ck^{d(1/2-1/q)+2/q-1}\|f\|_{L_2(N)} .
\eeq
\end{theorem}

\section{Proof of Theorem \ref{t3}}

For simplicity, denote
$H ((\pi + i \tau) b_1 + \xi)$
by $H(\tau)$. 
Let
$$
H_0(\tau)=\left.H(\tau)\right|_{V=0\mbox{, }\sigma=0}, \quad 
H_0=H_0(0) .
$$ 
The operator $H_0$ is a self-adjoint second-order elliptic differential operator 
on a manifold $M \times \T$. Let $E_k$ denote its spectral projector onto $[(k-1)^2;k^2)$. 
For a manifold $M$, we introduce

\vspace{4pt}
\noindent {\bf Condition} ${\mathrm{A}(q)}$. 
{\it M satisfies the property that for every
$\xi \in 
\R^m$,
$\<\xi,b_1\>=0$, there exist $\er>0$ and $C>0$ such that}
$$
\|E_k f\|_{L_{q}(M \times \T)}\leqslant 
C k^{1/2-\er} \|f\|_{L_2(M \times \T)}, \quad \forall f\in L_2(M\times\T).
$$
It is easy to see that ${\mathrm{A}(q)}$ implies $A(\tilde{q})$ if $\tilde{q}<q$. 

Let $\{\mu_j\}$ and $\{\varphi_j(x)\}$ be eigenvalues and eigenfunctions of the Laplace 
operator $-\Delta_x$ on $M$ with the corresponding (Dirichlet or Neumann) boundary conditions.
Then the eigenvalues of
$H_0(\tau)$ are of the form
$$
h_{j,n} (\tau) = |n + \pi b_1 + \xi|^2 + \mu_j - \tau^2 +
 2 i \tau \< n + \pi b_1, b_1 \> ,
$$
and the normalized eigenfunctions are
$$
\varphi_{j,n}(x,y)=|\Omega|^{-1/2}\varphi_j(x) e^{i\<n,y\>}, 
\quad j\in\N,\sk n\in\widetilde{\Gamma},
$$
where $\widetilde{\Gamma}$ is the dual lattice,
\begin{equation*}
\tilde \G =
 \left\{ n = \sum_{j=1}^m n_j \tilde b_j, \ n_j \in \Z \right\},
 \quad \< b_k, \tilde b_j \> = 2 \pi \de_{kj} .
\end{equation*}
Notice that $\<n,b_1\>\in 2\pi\Z$. This gives 
$$
|h_{j,n}(\tau)|\geqslant |\im h_{j,n}(\tau)|
= 2|\<n+\pi b_1,b_1\>||\tau|\geqslant 2\pi|\tau|.
$$  
Then, for $|\tau|>0$, the operator $H_0(\tau)$ is invertible and
\beq
\label{easy}
\|H_0(\tau)^{-1}\|\leqslant (2\pi|\tau|)^{-1},\quad \tau\neq 0.
\eeq
Consider also an operator $|H_0(\tau)|^{-1/2}$ such that
$$
|H_0(\tau)|^{-1/2} \varphi_{j,n} = |h_{j,n}(\tau)|^{-1/2} \varphi_{j,n}.
$$
The following Lemma is elementary.
\begin{lemma}
\label{lem}
Let $0<\eps<1/2$. Then the sums
\beq
\label{ms}
\sum_{k=1}^{\infty} \frac{k^{1-2\er}}{|k^2 - \tau^2|+|\tau|},\quad \sum_{k=1}^{\infty} \frac{k^{1-2\er}}{|(k-1)^2 - \tau^2|+|\tau|}
\eeq
are finite and uniformly bounded with respect to $\tau$ for $|\tau|>1$.
\end{lemma}
\begin{proof} 
For certainty, consider the first sum. Without loss of generality, we can assume $\tau>0$.
If $k^2\geqslant 2\tau^2$, then the denominator can be replaced with $\frac12 k^2$, 
and this implies that the "tail" of the sum converges uniformly.
Therefore, we may consider only $k^2<2\tau^2$. In this case,
$$
\sum_{k<2\tau}\frac{k^{1-2\er}}{|k^2 - \tau^2|+|\tau|}\leqslant 
2|\tau|^{1-2\eps}\sum_{k<2\tau}\frac{1}{|k^2 - \tau^2|+|\tau|}\leqslant 2\tau^{-2\eps}\sum_{k<2\tau}\frac{1}{|k-\tau|+1}.
$$
The last sum is bounded, because
$$
\tau^{-2\eps}\int\limits_0^{2\tau} \frac{dk}{|k - \tau| + 1} 
= 2\tau^{-2\eps}\int\limits_{\tau}^{2\tau}\frac{dk}{k-\tau+1}
= 2\tau^{-2\eps}\ln(\tau+1).\mbox{ }\qedhere
$$
\end{proof}

\begin{theorem}
\label{lq}
Assume that Condition ${\mathrm{A}(q)}$ holds.  
Then, for some $\tau_0>0$,
\beq
\l\||H_0(\tau)|^{-1/2}f\r\|_{L_q(M\times\T)}\leqslant C\|f\|_{L_2(M\times\T)},\quad \forall\:|\tau|>\tau_0,\,f\in L_2(M\times\T).
\eeq
\end{theorem}

\begin{proof}
Let $E_k$ be a spectral projector of $H_0$ onto $[(k-1)^2; k^2)$. 
Then
\begin{eqnarray*}
\l\||H_0(\tau)|^{-1/2}f\r\|_{L_q(M\times\T)}
\leqslant \sum_{k=1}^\infty \l\|E_k |H_0(\tau)|^{-1/2}f\r\|_{L_q(M\times\T)} \\
\leqslant 
C\sum_{k=1}^\infty k^{1/2-\er}\l\|E_k |H_0(\tau)|^{-1/2}f\r\|_{L_2(M\times\T)}
\leqslant 
C \sum_{k=1}^\infty k^{1/2-\er}
\l\|E_k |H_0(\tau)|^{-1/2}\r\| \cdot \|E_k f\|_{L_2(M\times\T)} ,
\end{eqnarray*}
from which, using Cauchy-Bunyakovsky-Schwarz inequality, we obtain
$$
\l\||H_0(\tau)|^{-1/2}f\r\|_{L_q(M\times\T)}^2 \leqslant
C \|f\|_{L_2(M\times\T)}^2 
\sum_{k=1}^\infty k^{1-2\varepsilon} \|E_k |H_0(\tau)|^{-1/2}\|^2 .
$$
The eigenvalues of $H_0$ are 
$|n + \pi b_1 + \xi|^2 + \mu_j$, 
$n\in \widetilde{\Gamma}$, $j\in \N$. 
The range of $E_k$ corresponds to the pairs $(j,n)$ such that
$(k-1)^2\le |n + \pi b_1 + \xi|^2 + \mu_j< k^2$. 
So,
\begin{multline*}
\l\|E_k |H_0(\tau)|^{-1/2}\r\|^2 =
\max_{|n + \pi b_1 + \xi|^2 + \mu_j\in [(k-1)^2;k^2)} 
\frac{1}{|h_{j, n} (\tau)|} \\
\leqslant \max_{|n + \pi b_1 + \xi|^2 + \mu_j\in [(k-1)^2;k^2)}
\frac{\sqrt{2}}{\l||n+\pi b_1+\xi|^2+\mu_j-\tau^2\r|+|\tau|}.
\end{multline*}
Finally, we need to show that the sum
\beq
\label{ev}
\sum_{k=1}^{\infty} \max_{|n + \pi b_1 + \xi|^2 + \mu_j\in [(k-1)^2;k^2)} 
\frac{k^{1-2\er}}
{\left| |n + \pi b_1 + \xi|^2 + \mu_j - \tau^2\r| +  |\tau|}
\eeq
is finite and uniformly bounded for $|\tau|>\tau_0$. 

To do this, we notice that in all the terms (maybe, all but one)
we can replace $|n + \pi b_1 + \xi|^2 + \mu_j$ 
with $(k-1)^2$ or $k^2$, and the term will not decrease, because, if $|\tau|\notin [k-1;k)$, 
then, after one of these substitutions, the denominator may only 
decrease. The term, for which $|\tau|\in [k-1;k)$, can be estimated 
by $C k^{-2\er}$ and does not affect the convergence. 
So, it is enough to consider two sums \eqref{ev}: 
we replace
$|n + \pi b_1 + \xi|^2 + \mu_j$ with $(k-1)^2$ in the first one,  
and with $k^2$ in the second one. Their boundness follows from Lemma \ref{lem}.
\end{proof}

We need the following fact to prove Theorem \ref{t3}:

\begin{lemma}
\label{auxl} 
Let $(M, \mu)$ be a measurable space with $\sigma$-finite measure, let
 $V \in L_p (M)$, $1 \le p < \infty$. 
Then for every $\de > 0$
there exists $c(\de)$ such that
$$
\int_M |V f g| d\mu \le
 \de \| f \|_{L_{2p'}(M)} \| g \|_{L_{2p'}(M)}
 + c (\de)  \| f \|_{L_2(M)} \| g \|_{L_2(M)},
 \quad f, g \in L_{2p'} (M),
$$
where $p'$ is the conjugate index to $p$.
\end{lemma}

\begin{proof}
The function $V$ can be expressed in the form
$$
V = V_1 + V_2, \quad \text{where} \quad
 \| V_1 \|_{L_p(M)} \le \de,
 \quad V_2 \in L_\infty (M).
$$
By H\"older inequality,
$$
\int_M |V f g| d\mu \le
 \de \| f \|_{L_{2p'}(M)} \| g \|_{L_{2p'}(M)}
 + \| V_2 \|_{L_\infty(M)}  \| f \|_{L_2(M)} \| g \|_{L_2(M)}.
\ \qedhere
$$
\end{proof}

\begin{theorem}
\label{result1}
Let $M$ satisfy ${\mathrm{A}(q)}$
for some $q \in (2, 2d / (d-2))$. 
Let $V \in L_p(M\times \T)$, where $p = q/(q-2)$. 
Then the operator $\l(H(\tau) - \la I\r)$ is invertible for
$|\tau| > \tau_0$, and
$\left\| \left( H (\tau) - \la I \right)^{-1} \right\| \le C |\tau|^{-1}$.
\end{theorem}

\begin{proof}
The condition on $V$ is invariant with respect to adding a constant.
So, without loss of generality, we can assume $\lambda=0$. 
It is enough to prove the following statement: 
for any $u\in \dom(H(\tau))$, $\|u\|_{L_2(M\times\T)}=1$, 
there exists $v\in \dom(H(\tau))$, $\|v\|_{L_2(M\times\T)}=1$, such that
\begin{equation*}
|(H(\tau)u,v)|\ge C|\tau|,\quad |\tau|>\tau_0.
\end{equation*}
Let $H_0(\tau)=\Phi_0(\tau)|H_0(\tau)|$ be the polar decomposition of $H_0(\tau)$.
We set
\beq
\label{vint}
v=\Phi_0(\tau)u. 
\eeq
Then,
\beq
\label{simple}
(H_0(\tau)u,v)=(|H_0(\tau)|u,u)\geqslant 2\pi|\tau|
\eeq
by \eqref{easy}, and
$$
(H_0(\tau)u,v)=\||H_0(\tau)|^{1/2}u\|^2_{L_2(M\times\T)}=\||H_0(\tau)|^{1/2}v\|^2_{L_2(M\times\T)}.
$$
Let us estimate the term $(Vu,v)$ using Lemma
\ref{auxl} and Theorem \ref{lq}:
\begin{multline}
\label{vest}
|(Vu,v)|\leqslant \delta \|u\|_{L_{q}(M\times\T)}\|v\|_{L_{q}(M\times\T)}+
c(\delta)\leqslant\\
\leqslant C\delta \||H_0(\tau)|^{1/2}v\|_{L_2(M\times\T)}\||H_0(\tau)|^{1/2}u\|_{L_2(M\times\T)} +c(\delta)=C\delta (H_0(\tau)u,v)+
c(\delta).
\end{multline}
This implies
$$
|(H(\tau)u,v)|\geqslant (1-C\delta)(H_0(\tau)u,v)-c(\delta)\geqslant 2\pi(1-C\delta)|\tau|-c(\delta)\geqslant C_1|\tau|
\quad \text{for} \ |\tau|>\tau_0,\ \delta<1/C.
$$
\end{proof}
\noindent {\it Proof of Theorem {\rm\ref{t3}}, the case of a manifold without boundary.}

If $\dd M = \emptyset$, then Corollary \ref{sl22} implies $\mathrm{A}(q)$ for all $q < 2d/(d-2)$.
From Theorem \ref{result1}, we get \eqref{18} for any $p>d/2$.
\qed

\noindent{\it Proof of Theorem {\rm\ref{t3}}, 
the case of Dirichlet or Neumann boundary conditions.}

If $k=1$ ($M$ is a line segment), then Theorem \ref{prod_t} again yields $\mathrm{A}(q)$ for all $q < 2d/(d-2)$.
And all $p>d/2$ are suitable.

If $d=3$, then Theorem \ref{mainit2} gives $\mathrm{A}(q)$ only if $q<6$,
so we need $p>3/2$.

If $d\geqslant 4$, then, again by Theorem \ref{mainit2}, Condition $\mathrm{A}(q)$ holds for 
 $q<(2d-4)/(d-3)$, and the corresponding condition on $V$ is
$V \in L_p(M\times \O)$, where $p > d - 2$.
\qed

To study the third type boundary condition, we use the following result from \cite{Su}.

\begin{theorem}
\label{su}
Let $k=1$, $M = [0, a]$, and assume that $\sigma$ satisfies \eqref{sigma}.
Then
\begin{eqnarray*}
\int\limits_{\O} |\sigma(0,y)|
\left|\l(|H_{0}(\tau)|^{-1/2}u\r)(0,y)\right|^2\,dy \\
+ \int\limits_{\O} |\sigma(a,y)|
\left|\l(|H_{0}(\tau)|^{-1/2}u\r)(a,y)\right|^2\,dy
\le \tilde{c}(\tau)\|u\|^2_{L_2([0,a]\times \O)},
\end{eqnarray*}
where $\lim\limits_{|\tau|\to\infty} \tilde{c}(\tau)=0$ uniformly over 
 $\xi'$ and $u\in L_2([0,a]\times\O)$.
\end{theorem}
\noindent {\it Proof of Theorem {\rm\ref{t3}}, 
the case of the third type boundary condition.}

Let $p>d/2$, $q = 2p' < 2d/(d-2)$.
Theorem \ref{prod_t} guaranties $\mathrm{A}(q)$.
Let $V \in L_p ([0,a]\times\O)$. 

For an arbitrary $u\in \dom (H_{\sigma})$, $\|u\|_{L_2([0,a]\times \O)}=1$,
let $v$ be defined by \eqref{vint}. 
Then
$$
(H_{\sigma}(\tau)u,v)=(H_0(\tau)u,v)+(Vu,v)+
\int\limits_{\O}\sigma(a,y)u(a,y)\overline{v}(a,y)\,dy -
\int\limits_{\O}\sigma(0,y)u(0,y)\overline{v}(0,y)\,dy .
$$
The first two terms are estimated in
\eqref{vest} and \eqref{simple}. 
Let us estimate the last one (the same can be done for the remaining term). Theorem \ref{su} gives
\begin{eqnarray*}
\left|\int\limits_{\O} \sigma(0,y)u(0,y)\overline{v}(0,y)\,dy \right|
\leqslant
\frac12 \int\limits_{\O} |\sigma(0,y)| 
\left(|u(0,y)|^2 + |v(0,y)|^2\right) \,dy \\
\leqslant \frac{\tilde{c}(\tau)}{2} 
\l( \l\| |H_0(\tau)|^{1/2}u\r\|^2_{L_2([0,a]\times\Omega)}
+ \l\| |H_0(\tau)|^{1/2}v\r\|^2_{L_2([0,a]\times\Omega)} \r)=
\tilde{c}(\tau)\l(H_0(\tau)u,v\r).
\end{eqnarray*}
Hence,
$$
\l|(H_{\sigma}(\tau)u,v)\r|\geqslant 
(H_0(\tau)u,v) \l(1-C\delta- 2\tilde{c}(\tau)\r)-c(\delta)
\geqslant 2\pi\l(1-C\delta-2\tilde{c}(\tau)\r)|\tau|-c(\delta),
\quad |\tau|>\tau_0,
$$
where $\delta$ and $\tau_0$ are chosen in such a way that 
 $C\delta+2\tilde{c}(\tau)<1$, $|\tau|>\tau_0$. 
The last estimate implies \eqref{18}.
\qed



\begin{thebibliography}{99}

\bibitem{BSu1}
Birman~M.~Sh., Suslina~T.~A., {\it Absolute continuity of a two-dimensional periodic magnetic Hamiltonian with discontinuous vector potential}, Algebra i Analiz 10 (1998), no. 4, p. 1--36. English translation in St. Petersburg Math. J. 10 (1999), no. 4, p. 579--601.

\bibitem{BSu}
Birman~M.~Sh, Suslina~T.~A., {\it Periodic magnetic Hamiltonian with variable metrics. Problem of absolute continuity}, Algebra i Analiz, vol. 11 (1999), 2, pp. 1-40. English translation in St. Petersburg Math. J. 11 (2000), no. 2, p. 203--232.

\bibitem{D}
Danilov~L.~I., 
{\it On absolute continuity of the spectrum of a periodic magnetic 
Schr\"odinger operator},
J. Phys. A: Math. Theor. 42 (2009) 275204. 

\bibitem{KF}
N.~Filonov, I.~Kachkovskii, {\it Absolute continuity of the spectrum of a periodic Schr\"odinger operator in a multidimensional cylinder}, Algebra i Analiz, 21 (2009), no. 1, p. 133-Ц152. English translation in St. Petersburg Math. J. 21 (2010), no. 1, p. 95--109.

\bibitem{K}
Kato~T., {\it Perturbation Theory for Linear Operators},
Grundlehren der mathematischen Wissenschaften, Vol. 132. Berlin-Heidelberg-New York: Springer-Verlag, 1966.

\bibitem{RS}
Reed~M., Simon~B., {\it Methods of Modern Mathematical Physics, Vol $4$: Analysis of Operators}, Academic Press, New-York, 1978.

\bibitem{SoSh}
Shargorodsky~E., Sobolev~A.~V., {\it Quasiconformal mappings and
periodic spectral problems in dimension two}, J. Anal. Math. 91
(2003), p. 67--103.

\bibitem{Shen}
Shen~Z., {\it On absolute continuity of the periodic Schr\"odinger
operators}, Intern. Math. Res. Notes (2001), no. 1, p. 1--31.

\bibitem{SuSh}
Shterenberg~R.G., Suslina~T.~A., {\it Absolute continuity of the spectrum of the Schr\"odinger operator with the potential concentrated on a periodic system of hypersurfaces}, Algebra i Analiz 13 (2001), no. 5, pp. 197--240. English translation in St. Petersburg Math. J. 13 (2002), no. 5, p. 859--891. 

\bibitem{SuSh2}
Shterenberg~R.~G., Suslina~T.~A., {\it Absolute continuity of the spectrum of the magnetic Schr\"odinger operator with a metric in a two-dimensional periodic waveguide}. Algebra i Analiz 14 (2002), no. 2, pp. 159--206. English translation in St. Petersburg Math. J. 14 (2003), no. 2, p. 305--343.

\bibitem{SmSo}
Smith~H.~F., Sogge~C.~D., {\it On the $L_p$ norm of spectral clusters 
for compact manifolds with boundary}, 
Acta Mathematica 198 (2007) no. 1, p. 107--153.

\bibitem{So}
Sogge~C.~D., {\it Concerning the $L^p$ norm of spectral clusters 
for second-order elliptic operators on compact manifolds}, 
J. Funct. Anal. 77 (1988) no. 1, p. 123--138.

\bibitem{Su}
Suslina~T.~A., {\it On the absence of eigenvalues of a periodic
matrix Schr\"odinger operator in a layer}, Russian Journal of
Mathematical Physics 8 (2001), no. 4, p. 463--486.

\bibitem{Th}
Thomas~L., {\it Time dependent approach to scattering 
from impurities in a crystal}, 
Commun. Math. Phys. 33 (1973), p. 335Ч-343.

\end{thebibliography}
\end{document}